\documentclass[]{amsart}
\usepackage[applemac]{inputenc}
\usepackage[]{fontenc}
\usepackage{amssymb}
\usepackage{setspace}
\usepackage{amsthm}
\usepackage{amsmath}
\usepackage{color}
\usepackage{esint}
\usepackage{mathtools}
\usepackage{esvect}
\usepackage{mathrsfs}
\newcommand{\de}{\partial}
\newcommand{\epsi}{\varepsilon}
\newcommand{\metricsobolev}{W^{1,p}(\Omega;X)}

\newcommand{\gengrad}[1]{\abs{\nabla#1}}
\newcommand{\gendir}[1]{\abs{\partial_\nu#1}}
\newcommand{\abs}[1]{{\left|{#1}\right|}}

\newcommand{\R}{\mathbb{R}}
 \newcommand{\N}{\mathbb{N}}

\newcommand{\ptdef}{\mathrel{\mathop:}=}

\def\XXint#1#2#3{{\setbox0=\hbox{$#1{#2#3}{\int}$}
     \vcenter{\hbox{$#2#3$}}\kern-.5\wd0}}
\def\Xint#1{\mathchoice
   {\XXint\displaystyle\textstyle{#1}}%
   {\XXint\textstyle\scriptstyle{#1}}%
   {\XXint\scriptstyle\scriptscriptstyle{#1}}%
   {\XXint\scriptscriptstyle\scriptscriptstyle{#1}}%
   \!\int}
\def\dashint{\Xint-}

\def\restriction#1#2{\mathchoice
              {\setbox1\hbox{${\displaystyle #1}_{\scriptstyle #2}$}
              \restrictionaux{#1}{#2}}
              {\setbox1\hbox{${\textstyle #1}_{\scriptstyle #2}$}
              \restrictionaux{#1}{#2}}
              {\setbox1\hbox{${\scriptstyle #1}_{\scriptscriptstyle #2}$}
              \restrictionaux{#1}{#2}}
              {\setbox1\hbox{${\scriptscriptstyle #1}_{\scriptscriptstyle #2}$}
              \restrictionaux{#1}{#2}}}
\def\restrictionaux#1#2{{#1\,\smash{\vrule height .8\ht1 depth .85\dp1}}_{\,#2}} 

\numberwithin{equation}{section}
\date{\today}

\newtheorem{thm}{Theorem}[section]
\newtheorem{proposition}[thm]{Proposition}

\newtheorem{lem}[thm]{Lemma}
\theoremstyle{definition} 
\newtheorem{definition}{Definition}[section]

\title[$p$-energy of metric space valued Sobolev maps]{A representation formula for the $p$-energy of metric space valued Sobolev maps}
\author{Philippe Logaritsch}
\address{Institut f\"ur Mathematik der Universit\"at Z\"urich,
Winterthurerstr.~190, CH-8057 Z\"urich (Switzerland)}
\email{logaritsch@access.uzh.ch}

\author{Emanuele Spadaro}
\address{Max-Planck-Institut f\"ur Mathematik in den Naturwissenschaften\\ Inselstr.~1, D-04103 Leipzig (Germany)}
\email{spadaro@mis.mpg.de}
\begin{document}

\begin{abstract}
We give an explicit representation formula for the $p$-energy of Sobolev maps with values in a metric space as defined by Korevaar and Schoen (Comm.~Anal.~Geom.~{\bf 1} (1993), no.~3-4, 561--659). The formula is written in terms of the Lipschitz compositions introduced by Ambrosio (Ann.~Scuola Norm.~Sup. Pisa Cl.~Sci.~(4) {\bf 3} (1990), n.~17, 439--478), thus further relating the two different definitions considered in the literature.
\end{abstract}

\maketitle

\section{Introduction}
In this short note we show an explicit representation formula for the $p$-energy of weakly differentiable maps with values in a separable complete metric space, thus giving a contribution to the equivalence between different theories considered in the literature.

Since the early 90's, weakly differentiable functions with values in singular spaces have been extensively studied in connection with several questions in mathematical physics and geometry (see, for instance, \cite{Ambrosio, Ambrosio_Kirchheim, Chiron, FHK, GrSc, Hajlasz, Heinonen, HKST, Jost, Jost2, Jost3, KS, KS_2, Ohta, Reshetnyak_1, Reshetnyak_2, Reshetnyak_3, Shanmugalingam}).
Among the different approaches which have been proposed, we recall here the ones by Korevaar and Schoen \cite{KS} and Jost \cite{Jost} based on two different expressions of approximate energies; that by Ambrosio \cite{Ambrosio} and Reshetnyak \cite{Reshetnyak_1} using the compositions with Lipschitz functions; the Newtonian--Sobolev spaces \cite{HKST}; and the Cheeger-type Sobolev spaces \cite{Ohta}.

As explained by Chiron \cite{Chiron}, all these notions coincide when the domain of definition is an open subset of $\R^n$ (or a Riemannian manifold) and the target is a complete separable metric space $X$ (contributions to the proof of these equivalences have been given in \cite{Chiron, HKST, Reshetnyak_2, Shanmugalingam}).
Nevertheless, the energies associated with these different approaches do not coincide in general (we refer again to \cite{Chiron} for a detailed discussion about the pairwise comparisons).
For instance, the Dirichlet energies defined by Korevaar and Schoen \cite{KS} and Jost \cite{Jost} are a generalization of the classical harmonic energy of maps with values in a Riemannian manifold; while the natural energy associated with the definition given by Reshetnyak does not coincide with the Dirichlet energy, but rather corresponds to the integral of the square of the operator norm of $\nabla u$ for maps with values in $\R^m$ (see, for instance, \cite{Reshetnyak_3}).
On the other hand, even if not the right generalization, the energy considered by Reshetnyak has a representation formula given by a supremum of compositions with Lipschitz functions, while the harmonic energy by Korevaar--Schoen and Jost is expressed by a limiting process which does not lead to an explicit formula.

In some applications an expression for the energy may be actually desirable. In a recent work by De Lellis and the second author \cite{DS1}, a connection between the two approaches has been found in the special case of functions with values in the metric space of multiple points. Indeed, a formula for the Dirichlet energy of Almgren's $Q$-valued functions in terms of Lipschitz compositions is in fact the starting point to revisit the regularity theory and develop a new approach.

Here we show that such a link can be found in general, i.e.~the $p$-energy $E^p(u)$ introduced by Korevaar and Schoen \cite{KS} can be actually expressed in terms of the compositions with the distance functions considered by Reshetnyak, thus leading to an explicit formula for the energy density (see the next sections for precise definitions).

\begin{thm}\label{t.main}
Let $(X,d)$ be a separable, complete metric space, $p\in[1,\infty[$ and $u\in W^{1,p}(\Omega;X)$, with $\Omega\subset\R^n$ open and bounded. Then, the $p$-energy of $u$ have the following explicit representation:
\begin{equation}\label{e.representation}
E^p(u) = \int_\Omega \left(\dashint_{S^{n-1}}\abs{\partial_\nu u}^p(x)\,\, \mathrm{d}\mathcal{H}^{n-1}(\nu)\right)\mathrm{d}\mathcal{L}^n(x),
\end{equation}
where
\begin{equation}\label{e.partial}
 \abs{\partial_\nu u}(x) = \sup_{\xi\in \mathscr{D}} |\nu\cdot\nabla (d( u(x), \xi))| \quad\text{for a.e. } x\in \Omega,
\end{equation}
and $\mathscr{D}\subset X$ is any countable dense subset.
\end{thm}

Theorem~\ref{t.main} is in line with the results in \cite{DS1}. However, if in the case of Almgren's multiple valued functions it is enough to sum the partial derivatives of the composition functions, in the case of a generic metric space an orthonormal frame may not be sufficient, but one needs instead to consider an average of all partial derivatives in all directions (see \S~\ref{ss.counter} for an example in which the partial derivatives do not suffice to give the harmonic energy).

\section{Sobolev Maps with Values in a Metric Space}
In this note we restrict ourself to consider the form of the $p$-energy as defined by Korevaar and Schoen \cite{KS}, the equivalence with the definition by Jost \cite{Jost} being shown in \cite{Chiron}.
Moreover, for the sake of simplicity in the exposition, we consider here only the case of maps with domain a bounded open subset $\Omega\subset\R^n$. Indeed, everything can be easily generalized to the case of domains in a Riemannian manifold $(M,g)$.

\medskip

In what follows, $(X,d)$ is a complete, separable metric space.
We denote by $L^p(\Omega;X)$ the set of measurable functions $u:\Omega\to X$ such that, for some (and hence every) $\xi \in X$, $x\mapsto d(u(x),\xi)$ is a function in $L^p(\Omega)$.
For simplicity of notation, in the following we consider the case $p\in [1,\infty[$.
$W^{1,\infty}$ maps with values in $X$ are exactly the Lipschitz maps and all the results below are actually simpler in this case.

In \cite{KS} the authors proposed a definition of Sobolev maps into a metric space starting from a family of approximate energies. The following is an equivalent, but for some aspects  simplified, definition.

\begin{definition}\label{def:harmonic_energy_according_to_KS}
Let $p\in\,[1,\infty[$ and $u\in L^p(\Omega;X)$. The \emph{Korevaar--Schoen $p$-energy} of $u$ is given by
\begin{equation*}
E^p(u)\ptdef \sup_{\substack{f\in C_c(\Omega)\\0\leq f\leq 1}}\left(
\limsup_{h \to 0} \int_{\Omega}e_{h,p}^{u}(x)f(x)\mathrm{d}x \right),
\end{equation*}
with
\begin{equation}
\label{eq:approx_energy}
e_{h,p}^{u}(x)\ptdef\begin{cases} \displaystyle c_{n,p}\int_{B_1(0)}	\frac{d^p(u(x),u(x+hv))}{h^p}\,\,\mathrm{d}\mathcal{L}^{n}(v)&	\mbox{if}\,\, x\in \Omega_h,\\
0 & \text{else},
\end{cases}
\end{equation}
where $c_{n,p}\ptdef \frac{n+p}{n\,\omega_n}$ and $\Omega_h\ptdef\{x\in \Omega \, : \, d(x,\de\Omega)>h\}$.
A map $u$ is said to belong to the Sobolev space $W^{1,p}(\Omega;X)$ for $p\in]1,\infty[$, or to the space of functions with bounded variation $BV(\Omega;X)$ for $p=1$, if $E^p(u)<+\infty$.
\end{definition}

Note that, unlike in the original paper \cite{KS}, we did not base the definition of $E^p$ on spherical averages but on ball averages -- and we also divided by $n\,\omega_n$ in order to save us from taking care of different normalization factors hereafter. 
Korevaar and Schoen proved that, when $E^p(u)<+\infty$, the measures $e_{h,p}^{u}\mathrm{d}\mathcal{L}^n$ converge weakly as $h\to0$ to the same limit measure $\mu$ as the spherical averages do \cite[Theorem 1.5.1]{KS}.
Furthermore, it is proved in \cite[Theorem~1.10]{KS} that, for $p\in]1,\infty[$, $\mu$ is absolutely continuous with respect to $\mathcal{L}^n$, i.e.~there exists some $h\in L^1(\Omega)$ such that
\begin{equation}\label{e.g}
\mu=h \,\mathrm{d}\mathcal{L}^n.
\end{equation}
In the case $p=1$, if the limit measure is absolutely continuous with respect to $\mathcal{L}^n$, then $u$ is said to belong to the Sobolev space $W^{1,1}(\Omega;X)$.

\medskip

The spaces $W^{1,p}(\Omega;X)$ can also be characterized by using the composition with Lipschitz functions of the metric space following the approach by Ambrosio \cite{Ambrosio} and Reshetnyak \cite{Reshetnyak_1}.
Indeed, as proven in \cite{Reshetnyak_2} (see also \cite[Proposition 4]{Chiron}) the following holds.

\begin{proposition}\label{p:harmonic_energy_according_to_R}
Let $p\in[1,\infty[$. Then, $u\in W^{1,p}(\Omega;X)$ if and only if
\begin{itemize}
\item[(i)] for every $\xi\in X$, the map $\Omega\ni x\mapsto d(u(x),\xi)$ belongs to $W^{1,p}(\Omega)$;
\item[(ii)] there exists $g\in L^p(\Omega)$ such that, for every $\xi\in X$,
\begin{equation}\label{e.grad}
\abs{\nabla (d(u,\xi))}\leq g\quad \mathcal{L}^n\text{-a.e.}
\end{equation}
\end{itemize}
\end{proposition}

As shown by Reshetnyak \cite[Theorem 5.1]{Reshetnyak_1} (see also \cite[Proposition~4.2]{DS1} for the case of multiple valued functions, the proof remaining unchanged in the general case), there exists a minimal $g_{min}\in L^p(\Omega)$ such that (ii) holds: namely, if $g\in L^p(\Omega)$ satisfies \eqref{e.grad}, then $g_{min}\leq g$ $\mathcal{L}^n$-a.e. Moreover $g_{min}$ is given by the following expression:
\begin{equation}\label{eq:Resh_gradient}
g_{min}\ptdef \sup_{\xi\in\mathscr{D}}\abs{\nabla (d(u,\xi))},
\end{equation}
where $\mathscr{D}\subset X$ is any countable dense set.

Reshetnyak \cite{Reshetnyak_3} showed that in general the $p$-energy $E^p(u)$ does not coincide with $\|g_{\min}\|_{L^p}^p$.
In fact, in the case of $X=\R^m$, $E^2(u)$ equals the usual Dirichlet energy of the map $u$, while $\|g_{\min}\|_{L^2}^2$ is the integral of the square of the operator norm of $\nabla u$.
In the next section we show how to express the $p$-energy in terms of a variant of the supremum in \eqref{eq:Resh_gradient}.

\section{A representation formula for the $p$-energies}
Here we show how to recover the Korevaar--Schoen energy $E^p$ in the framework of Reshetnyak, proving Theorem~\ref{t.main}.

\subsection{Directional Derivatives}
To this purpose, we start showing the existence of a well-defined notion of the modulus of directional derivatives.

\begin{lem}
\label{lem:properties_generalized_directional_derivative}
Suppose $p\in\,[1,\infty[$, $\nu\in S^{n-1}$ and $u\in\metricsobolev$. Then, there exists a unique $g_{\nu}\in L^p(\Omega)$ such that
\begin{enumerate}
\item[(i)] for every $\xi \in X$, $\abs{\nu\cdot\nabla (d(u,\xi))}\leq g_\nu \ \mathcal{L}^n$-a.e.~in $\Omega$,
\item[(ii)] if $f\in L^p(\Omega)$ is such that, for every $\xi\in X$, $\abs{\nu\cdot\nabla (d(u,\xi))}\leq f \ \mathcal{L}^n$-a.e., then $g_\nu\leq f \ \mathcal{L}^n$-a.e. in $\Omega$.
\end{enumerate}
Moreover, the function $g_\nu$ is given by the following representation formula:
\begin{equation}
\label{eq:directional_derivative}
g_\nu\ptdef \sup_{\xi\in\mathscr{D}}\abs{\nu \cdot \nabla (d(u,\xi))},
\end{equation}
where $\mathscr{D}\subset X$ is any countable dense subset, and the map
\begin{eqnarray}\label{e.meas}
%H^p_u :  
\Omega\times S^{n-1}&\to& \R,\notag\\
(x,\nu) & \mapsto& g_\nu(x),
\end{eqnarray}
belongs to $L^p(\Omega\times S^{n-1})$ for the product measure $\mathcal{L}^n\times\mathcal{H}^{n-1}$.
\end{lem}

\begin{proof}
The proof follows closely the arguments in \cite[Proposition~4.2]{DS1}.
The uniqueness is an immediate consequence of (i) and (ii). Hence, it suffices to verify that the functions $g_\nu$ defined in \eqref{eq:directional_derivative} satisfy (i) and (ii).
The latter condition follows immediately from $\abs{\nu\cdot\nabla (d(u,\xi))}\leq f$ by taking the supremum in $\mathscr{D}$.
For (i), let $(\xi_k)_{k\in\N}$ in $\mathscr{D}$ converging to $\xi$. Then $(d(u,\xi_k))_{k\in\N}$ converges in $L^p(\Omega)$ to $d(u,\xi)$ and for every $\psi\in C^\infty_c(\Omega)$:
\begin{align*}
\int_\Omega \! (\nu\cdot\nabla d(u,\xi))\,\psi\,\mathrm{d}\mathcal{L}^n & = -\int_\Omega \! d(u,\xi)\,(\nu\cdot\nabla\psi)\,\mathrm{d}\mathcal{L}^n\\
&=-\lim_{k\to\infty} \int_\Omega \! d(u,\xi_k)\,(\nu\cdot\nabla\psi)\,\mathrm{d}\mathcal{L}^n\\
&= \lim_{k\to\infty} \int_\Omega \! (\nu\cdot\nabla d(u,\xi_k))\,\psi\,\mathrm{d}\mathcal{L}^n\\
&\leq\int_\Omega \! g_\nu\,\abs{\psi}\,\mathrm{d}\mathcal{L}^n.
\end{align*}
Since $\psi$ is arbitrary, we can deduce the desired inequality.

The last part of the statement simply follows from the measurability of the maps $(x,\nu) \mapsto \abs{\nu\cdot\nabla (d(u(x),\xi))}$ for every $\xi\in \mathscr{D}$ and the bound $g_\nu \leq g_{\min}\; \mathcal{L}^n$-a.e.~for every $\nu\in S^{n-1}$, where $g_{\min}$ is given in \eqref{eq:Resh_gradient}.
\end{proof}

In the sequel we denote by $\gendir{u}$ the function $g_\nu$ from the previous lemma, which will be the building blocks in order to find an expression for the Korevaar--Schoen $p$-energies.
Before giving the proof of Theorem~\ref{t.main}, we introduce this further notation: we set $\abs{\partial_vu}\ptdef\sup_{\xi\in\mathscr{D}}\abs{v\cdot \nabla (d(u,\xi))}$ for every $v\in \R^n\setminus \{0\}$ and notice that
\[
|\de_v u| = |v|\,|\de_{\frac{v}{|v|}} u|. 
\]
One checks immediately that, for $p\in \,[1,\infty[$ and $u\in\metricsobolev$,  it holds
\[
\dashint_{S^{n-1}}\abs{\partial_\nu u}^p\,\, \mathrm{d}\mathcal{H}^{n-1}(\nu)
=c_{n,p}\int_{B_1(0)}\!\abs{\partial_vu}^p\mathrm{d}\mathcal{L}^n(v)
\quad\text{for $\mathcal{L}^n$-a.e.~$x \in \Omega$,}
\]
where $c_{n,p}$ is the constant in \eqref{eq:approx_energy}.

We start premising the following lemma.

\begin{lem}\label{lem:key_estimate}
Let $p\in \,[1,\infty[$, $u \in W^{1,p}(\Omega;X)$, $v \in B_1(0)$ and $h>0$.
Then,
\begin{equation}\label{e.incremental}
\int_{\Omega_h} d^p(u(x+h \, v),u(x))
\mathrm{d}\mathcal{L}^n(x) \leq h^p \int_\Omega \abs{\de_v u}^p \mathrm{d}\mathcal{L}^n.
\end{equation}
\end{lem}

\begin{proof}
For every $\xi \in X$ and $\mathcal{L}^n\textit{-a.e. } x$ in $\Omega_h$, by the differentiability of Sobolev functions on almost every line, it holds
\begin{align}\label{e.direct incr}
\left\vert d(u(x+h\, v),\xi) - d(u(x), \xi) \right\vert^p & =  \left \vert \int_0^1 h\,v \cdot \nabla d(u(x+ t\, h\, v),\xi)\,\mathrm{d}t \right\vert^p \notag\\
 & \leq h^p \int_0^1 |\de_v u|^p(x+ t\, h\, v)\,\mathrm{d}t.
\end{align}
Since for every countable dense $\mathscr{D}\subset X$,
\[
d(u(x+h\,v),u(x)) = \sup_{\xi\in\mathscr{D}}\abs{d(u(x+h\,v),\xi)-d(u(x),\xi)} \quad \text{for } \mathcal{L}^n\textit{-a.e. } x\in\Omega_h,
\]
we infer from \eqref{e.direct incr} that for $\mathcal{L}^n\textit{-a.e. } x$ in $\Omega_h$:
\begin{equation}\label{e.direct incr 2}
 d^p(u(x+h\,v),u(x)) \leq h^p \int_0^1 |\de_v u|^p(x+ t\, h\, v)\,\mathrm{d}t.
\end{equation}
Integrating over $x$ and applying Fubini's theorem, we deduce easily \eqref{e.incremental}.
\end{proof}

\subsection{Proof of Theorem~\ref{t.main}}
For $p\in \,[1,\infty[$ and $u\in\metricsobolev$, we set
\begin{align*}
\mathcal{E}^p(u)&\ptdef	\int_\Omega
\dashint_{S^{n-1}}\abs{\partial_\nu u}^p\,\, \mathrm{d}\mathcal{H}^{n-1}(\nu)\,\mathrm{d}\mathcal{L}^n\\
&= c_{n,p} \int_\Omega \int_{B_1(0)}\!\abs{\partial_v u}^p\mathrm{d}\mathcal{L}^n(v)\,\mathrm{d}\mathcal{L}^n.
\end{align*}

\medskip

In order to prove Theorem~\ref{t.main}, we need to show that $E^p(u) = \mathcal{E}^p(u)$. We proceed in two steps.

{\sc Step 1}: $E^p(u)\leq\mathcal{E}^p(u)$.
Fix some $h>0$ and $f\in C_c(\Omega)$ with $0\leq f \leq 1$. Then,
\begin{align*}
\int_{\Omega}e_{h,p}^{u}(x)\,f(x)\,\mathrm{d}\mathcal{L}^n(x) &=
\int_{\Omega_h}\!\!\!c_{n,p}\int_{B_1(0)} \frac{d^p(u(x+hv),u(x))}{h^p}\,\,\mathrm{d}\mathcal{L}^n(v)\,f(x)\,\mathrm{d}\mathcal{L}^n(x)\\
&\leq c_{n,p} \int_{\Omega_h} \int_{B_1(0)} \frac{d^p(u(x+hv),u(x))}{h^p}\,\,\mathrm{d}\mathcal{L}^n(v)\,\mathrm{d}\mathcal{L}^n(x)\\
&=c_{n,p}\int_{B_1(0)} \int_{\Omega_h} \frac{d^p(u(x+hv),u(x))}{h^p}\,\,\mathrm{d}\mathcal{L}^n(x)\,\mathrm{d}\mathcal{L}^n(v),
\end{align*}
where we used Fubini's theorem in the last equality. Hence, by \eqref{e.incremental} in Lemma~\ref{lem:key_estimate}, we can infer
\begin{align*}
\int_{\Omega}e_{h,p}^{u}(x)\,f(x)\,\mathrm{d}\mathcal{L}^n(x) 
&\stackrel{\hidewidth \eqref{e.incremental} \hidewidth}{\leq}c_{n,p}			\int_{B_1(0)}\int_{\Omega} \abs{\partial_v u}^p(x) \,\mathrm{d}\mathcal{L}^n(x)\,
\mathrm{d}\mathcal{L}^n(v)\\
&=\mathcal{E}^p(u).
\end{align*}
As $\mathcal{E}^p(u)$ is independent of $h$ and $f$, we get the desired inequality by passing into the limit in $h\to 0$ and then taking the supremum over all $f\in C_c(\Omega)$ with $0\leq f \leq 1$. 

\medskip

{\sc Step 2}: $\mathcal{E}^p(u) \leq E^p(u)$.
For every $\epsi>0$, we fix $h_\epsi>0$ small enough to have
\begin{displaymath}
\mathcal{E}^p(u)\leq \int_{\Omega_{h_\epsi}}\!\!\!\!\!
c_{n,p} \int_{B_1(0)} \abs{\partial_v u}^p(x) \,\mathrm{d}\mathcal{L}^n(v) 	\mathrm{d}\mathcal{L}^n(x) + \epsi.
\end{displaymath}
Then, we pick $f\in C_c(\Omega)$ with $0\leq f\leq 1$ and $\restriction{f}{\Omega_{h_\epsi}}=1$ (which exists by Urysohn's lemma -- see, for instance, \cite[Lemma~2.12]{Rudin}). It follows that
\begin{align}\nonumber
\mathcal{E}^p(u)\leq& \int_{\Omega_{h_\epsi}}\!\!c_{n,p}\left(
\int_{B_1(0)} \abs{\partial_v u}^p(x) \mathrm{d}\mathcal{L}^n(v)
\right)f(x)\,\mathrm{d}\mathcal{L}^n(x) + \epsi\\
=& \int_{\Omega_{h_\epsi}}\!\!c_{n,p}\left(
\int_{B_1(0)} \sup_{k\in\N}\abs{v\cdot\nabla (d(u(x),\xi_k))}^p \mathrm{d}\mathcal{L}^n(v)
\right)f(x)\,\mathrm{d}\mathcal{L}^n(x) + \epsi,
\label{anfang}
\end{align}
where $\{\xi_k\}_{k\in\N}\subset X$ is any dense subset.
Using monotone convergence, we can rewrite the interior integral in the following way: for $\mathcal{L}^n\textit{-a.e. } x$ in $\Omega_{h_\epsi}$,
\begin{multline*}
\int_{B_1(0)} \sup_{k\in\N}\abs{v\cdot\nabla (d(u(x), \xi_k))}^p \mathrm{d}\mathcal{L}^n(v)\\
=\lim_{N\to\infty}
\int_{B_1(0)} \max_{1\leq k\leq N}\abs{v\cdot\nabla (d(u(x), \xi_k))
}^p \mathrm{d}\mathcal{L}^n(v).
\end{multline*}
On the other hand, by the $L^p$-approximate differentiability of Sobolev functions (see, for example, \cite[6.1.2]{EG}), we know that, for $\mathcal{L}^n\textit{-a.e. } x$ in $\Omega_{h_\epsi}$, the incremental quotients
\[
R_{h,k}(v):=\frac{d(u(x+h\,v),\xi_k) - d(u(x),\xi_k)}{h},
\]
converge in $L^p(B_1(0))$ as $h\to0$ to the linear function $L_k(v) = v\cdot\nabla (d(u(x), \xi_k))$. Therefore, it follows that
\begin{align*}
\int_{B_1(0)} \max_{1\leq k\leq N}\abs{v\cdot\nabla (d(u(x), \xi_k))
}^p \mathrm{d}\mathcal{L}^n(v) = \lim_{h\to 0} \int_{B_1(0)} \max_{1\leq k\leq N}|R_{h,k}(v)|^p\,\mathrm{d}\mathcal{L}^n(v)\\
\leq \liminf_{h\to 0} \int_{B_1(0)}\abs{\frac{d(u(x+hv),u(x))}{h}}^p \mathrm{d}\mathcal{L}^n(v),
\end{align*}
where we used the triangle inequality
\[
d(u(x+h\,v),\xi_k) - d(u(x),\xi_k)\leq d(u(x+h\,v),u(x)).
\]

Combining this estimate with inequality (\ref{anfang}), we deduce:
\begin{align*}	
\mathcal{E}^p(u)&\leq \int_{\Omega_{h_\epsi}}\!\!c_{n,p}
\liminf_{h\to 0} \left(\int_{B_1(0)}\abs{\frac{d(u(x+hv),u(x))}{h}}^p \mathrm{d}\mathcal{L}^n(v)\right)
f(x)\,\mathrm{d}\mathcal{L}^n(x) + \epsi\\
&\stackrel{\hidewidth(\ast)\hidewidth}{\leq} \liminf_{ h\to 0}
\int_{\Omega_{h_\epsi}}\!\!c_{n,p}
\int_{B_1(0)}\abs{\frac{d(u(x+hv),u(x))}{h}}^p \mathrm{d}\mathcal{L}^n(v)
f(x)\,\mathrm{d}\mathcal{L}^n(x) + \epsi,\\
&\leq  \liminf_{h\to 0}
\int_{\Omega}\!\!
e^u_{h,p}(x)\,f(x)\,\mathrm{d}\mathcal{L}^n(x) + \epsi\\
% \leq& \sup_{\substack{f\in C_c(\Omega)\\0\leq f\leq 1}}\left( 
% \limsup_{h\to 0}
% \int_{\Omega}\!
% e^u_{h,p}(x)\,f(x)\,\mathrm{d}\mathcal{L}^n(x)\right) + \epsi\\
& \leq E^p(u)+\epsi,
\end{align*}
where we used Fatou's lemma in $(\ast)$.
As $\epsi>0$ was arbitrary, we deduce the desired inequality.
\qed

Theorem~\ref{t.main} gives an explicit representation formula for the $p$-energy of any function $u\in\metricsobolev$, as well for the limiting energy density $h$ by Korevaar and Schoen in \eqref{e.g}. Indeed, we can localize the equality between the energies in any subdomain $\Omega'\subset\Omega$, thus leading to the equality between the densities:
\[
h = \dashint_{S^{n-1}}\abs{\partial_\nu u}^p\,\, \mathrm{d}\mathcal{H}^{n-1}(\nu).
\]

\subsection{An example}\label{ss.counter}
As explained in the Introduction, a first special instance of a formula for the Dirichlet energy in terms of Lipschitz compositions is the one for Almgren's multiple valued functions in \cite{DS1}.
In this case, indeed, it is enough to sum the values of the directional derivatives along an orthonormal frame.
Here we show that in general this is not enough.

To this purpose, consider $X\ptdef(\R^2,d_\infty)$, where $d_\infty$ is the distance induced by the maximum norm, i.e.~ for any $\xi=(\xi_1,\xi_2)$, $\eta=(\eta_1,\eta_2)\in\R^2$,
\[
d_\infty(\xi,\eta)\ptdef\max\{\abs{\xi_1-\eta_1},\abs{\xi_2-\eta_2}\}. 
\]
The following elementary lemma shows that a map whose components are classical Sobolev functions belongs actually to $W^{1,p}(\Omega;X)$.

\begin{lem}\label{lem:aux_for_counterex}
Let $\Omega\subset\R^n$ be an open, bounded set, and $f_1,f_2\in W^{1,p}(\Omega)$.
Then, the map $u\ptdef(f_1,f_2)$ belongs to $W^{1,p}(\Omega;X)$ and, for every $\nu \in S^1$, 
\[
\gendir{u}(x)=\max\{\abs{\nu\cdot\nabla f_1(x)},\abs{\nu\cdot\nabla f_2(x)}\} \quad \text{for } \;\mathcal{L}^n\text{-a.e.} \; x\in\Omega. 
\]
\end{lem}

\begin{proof}
Fix any $\xi=(\xi_1,\xi_2)\in\R^2$. As $d(u(x),\xi)=\max\{\abs{f_1(x)-\xi_1},\abs{f_2(x)-\xi_2}\}$ for $\mathcal{L}^n$-a.e. $x\in\Omega$, we see that $d(u,\xi)$ belongs to $W^{1,p}(\Omega)$. Since for $\mathcal{L}^n$-a.e. $x\in\Omega$ we have:

\begin{equation}
\label{eqn:gendir_formula}
|\nabla (d(u(x),\xi))|=\begin{cases}
	|\nabla f_1(x)|&	\mbox{if}\,\, \abs{f_1(x)-\xi_1}> \abs{f_2(x)-\xi_2},\\
	|\nabla f_2(x)|	&	\mbox{else},
\end{cases}
\end{equation}
one can estimate $\abs{\nabla d(u,\xi)}$ by $g\ptdef\max\{\gengrad{f_1},\gengrad{f_2}\}$. As $g$ is independent of $\xi$ and since $g$ belongs to $L^p(\Omega)$ we infer that $u$ belongs to $W^{1,p}(\Omega;X)$. Moreover, we see that the remaining claim follows also from $(\ref{eqn:gendir_formula})$. 
\end{proof}

Now, using the previous lemma, we can easily find an  example where the sum of the partial derivatives on a frame does not equal the harmonic energy.
Consider, for instance, the map $u\ptdef(f_1,f_2):\Omega\subset\R^2\to\R^2$, where $f_i(x_1,x_2)\ptdef x_i$ for $i\in\{1,2\}$. Applying Lemma \ref{lem:aux_for_counterex}, we get $\gendir{u}=\max\{\abs{\nu_1},\abs{\nu_2}\}$. Hence, we infer that
\[
2=\abs{\partial_{e_1} u}^2(x)+\abs{\partial_{e_2} u}^2(x) > \dashint_{S^1}\gendir{u}^2(x) \, \mathrm{d}\mathcal{H}^1(\nu)=\frac{2+\pi}{2\,\pi}.
\]

\nocite{*}
\bibliographystyle{plain}
\bibliography{biblio-metric}

\begin{thebibliography}{10}

\bibitem{Ambrosio}
Luigi Ambrosio.
\newblock Metric space valued functions of bounded variation.
\newblock {\em Ann. Scuola Norm. Sup. Pisa Cl. Sci. (4)}, 17(3):439--478, 1990.

\bibitem{Ambrosio_Kirchheim}
Luigi Ambrosio and Bernd Kirchheim.
\newblock Currents in metric spaces.
\newblock {\em Acta Math.}, 185(1):1--80, 2000.

\bibitem{Chiron}
David Chiron.
\newblock On the definitions of {S}obolev and {BV} spaces into singular spaces
  and the trace problem.
\newblock {\em Commun. Contemp. Math.}, 9(4):473--513, 2007.

\bibitem{DS1}
Camillo De~Lellis and Emanuele~Nunzio Spadaro.
\newblock {$Q$}-valued functions revisited.
\newblock {\em Mem. Amer. Math. Soc.}, 211(991):vi+79, 2011.

\bibitem{EG}
Lawrence~C. Evans and Ronald~F. Gariepy.
\newblock {\em Measure theory and fine properties of functions}.
\newblock Studies in Advanced Mathematics. CRC Press, Boca Raton, FL, 1992.

\bibitem{FHK}
B.~Franchi, P.~Haj{\l}asz, and P.~Koskela.
\newblock Definitions of {S}obolev classes on metric spaces.
\newblock {\em Ann. Inst. Fourier (Grenoble)}, 49(6):1903--1924, 1999.

\bibitem{GrSc}
Mikhail Gromov and Richard Schoen.
\newblock Harmonic maps into singular spaces and {$p$}-adic superrigidity for
  lattices in groups of rank one.
\newblock {\em Inst. Hautes \'Etudes Sci. Publ. Math.}, (76):165--246, 1992.

\bibitem{Gromov_Schoen}
Mikhail Gromov and Richard Schoen.
\newblock Harmonic maps into singular spaces and {$p$}-adic superrigidity for
  lattices in groups of rank one.
\newblock {\em Inst. Hautes \'Etudes Sci. Publ. Math.}, (76):165--246, 1992.

\bibitem{Hajlasz}
Piotr Haj{\l}asz.
\newblock Sobolev mappings between manifolds and metric spaces.
\newblock In {\em Sobolev spaces in mathematics. {I}}, volume~8 of {\em Int.
  Math. Ser. (N. Y.)}, pages 185--222. Springer, New York, 2009.

\bibitem{Heinonen}
Juha Heinonen.
\newblock Nonsmooth calculus.
\newblock {\em Bull. Amer. Math. Soc. (N.S.)}, 44(2):163--232, 2007.

\bibitem{HKST}
Juha Heinonen, Pekka Koskela, Nageswari Shanmugalingam, and Jeremy~T. Tyson.
\newblock Sobolev classes of {B}anach space-valued functions and quasiconformal
  mappings.
\newblock {\em J. Anal. Math.}, 85:87--139, 2001.

\bibitem{Jost}
J{\"u}rgen Jost.
\newblock Equilibrium maps between metric spaces.
\newblock {\em Calc. Var. Partial Differential Equations}, 2(2):173--204, 1994.

\bibitem{Jost2}
J{\"u}rgen Jost.
\newblock Convex functionals and generalized harmonic maps into spaces of
  nonpositive curvature.
\newblock {\em Comment. Math. Helv.}, 70(4):659--673, 1995.

\bibitem{Jost3}
J{\"u}rgen Jost.
\newblock Generalized {D}irichlet forms and harmonic maps.
\newblock {\em Calc. Var. Partial Differential Equations}, 5(1):1--19, 1997.

\bibitem{KS}
Nicholas~J. Korevaar and Richard~M. Schoen.
\newblock Sobolev spaces and harmonic maps for metric space targets.
\newblock {\em Comm. Anal. Geom.}, 1(3-4):561--659, 1993.

\bibitem{KS_2}
Nicholas~J. Korevaar and Richard~M. Schoen.
\newblock Global existence theorems for harmonic maps to non-locally compact
  spaces.
\newblock {\em Comm. Anal. Geom.}, 5(2):333--387, 1997.

\bibitem{Ohta}
Shin-ichi Ohta.
\newblock Cheeger type {S}obolev spaces for metric space targets.
\newblock {\em Potential Anal.}, 20(2):149--175, 2004.

\bibitem{Reshetnyak_1}
Yu.~G. Reshetnyak.
\newblock Sobolev classes of functions with values in a metric space.
\newblock {\em Sibirsk. Mat. Zh.}, 38(3):657--675, iii--iv, 1997.

\bibitem{Reshetnyak_2}
Yu.~G. Reshetnyak.
\newblock Sobolev-type classes of functions with values in a metric space. ii.
\newblock {\em Siberian Mathematical Journal}, 45:709--721, 2004.

\bibitem{Reshetnyak_3}
Yu.~G. Reshetnyak.
\newblock To the theory of sobolev-type classes of functions with values in a
  metric space.
\newblock {\em Siberian Mathematical Journal}, 47:117--134, 2006.
\newblock 10.1007/s11202-006-0013-x.

\bibitem{Rudin}
Walter Rudin.
\newblock {\em Real and complex analysis}.
\newblock McGraw-Hill Book Co., New York, third edition, 1987.

\bibitem{Shanmugalingam}
Nageswari Shanmugalingam.
\newblock Newtonian spaces: an extension of {S}obolev spaces to metric measure
  spaces.
\newblock {\em Rev. Mat. Iberoamericana}, 16(2):243--279, 2000.

\end{thebibliography}

\end{document}